\documentclass[fleqn,reqno,11pt,a4paper,final]{amsart}

\usepackage[a4paper,left=30mm,right=30mm,top=30mm,bottom=30mm,marginpar=20mm]{geometry} 
\usepackage{amsmath}
\usepackage{amssymb}
\usepackage{amsthm}
\usepackage{amscd}
\usepackage[ansinew]{inputenc}
\usepackage{cite}
\usepackage{bbm}
\usepackage{color}
\usepackage[english=american]{csquotes}
\usepackage[final]{graphicx}
\usepackage{hyperref}
\usepackage{calc}
\usepackage{mathptmx}

\graphicspath{{../Pictures/}}

\numberwithin{equation}{section}

\newtheoremstyle{thmlemcorr}{10pt}{10pt}{\itshape}{}{\bfseries}{.}{10pt}{{\thmname{#1}\thmnumber{ #2}\thmnote{ (#3)}}}
\newtheoremstyle{thmlemcorr*}{10pt}{10pt}{\itshape}{}{\bfseries}{.}\newline{{\thmname{#1}\thmnumber{ #2}\thmnote{ (#3)}}}
\newtheoremstyle{remexample}{10pt}{10pt}{}{}{\bfseries}{.}{10pt}{{\thmname{#1}\thmnumber{ #2}\thmnote{ (#3)}}}
\newtheoremstyle{ass}{10pt}{10pt}{}{}{\bfseries}{.}{10pt}{{\thmname{#1}\thmnumber{ A#2}\thmnote{ (#3)}}}

\theoremstyle{thmlemcorr}
\newtheorem{theorem}{Theorem}
\numberwithin{theorem}{section}
\newtheorem{lemma}[theorem]{Lemma}
\newtheorem{corollary}[theorem]{Corollary}
\newtheorem{proposition}[theorem]{Proposition}

\newtheorem{definition}[theorem]{Definition}

\theoremstyle{thmlemcorr*}
\newtheorem{theorem*}{Theorem}
\newtheorem{lemma*}[theorem]{Lemma}
\newtheorem{corollary*}[theorem]{Corollary}
\newtheorem{proposition*}[theorem]{Proposition}
\newtheorem{problem*}[theorem]{Problem}
\newtheorem{conjecture*}[theorem]{Conjecture}
\newtheorem{definition*}[theorem]{Definition}

\theoremstyle{remexample}
\newtheorem{remark}[theorem]{Remark}

\theoremstyle{ass}

\newcommand{\Crm}{\mathrm{C}}

\newcommand{\Lrm}{\mathrm{L}}

\newcommand{\Wrm}{\mathrm{W}}

\newcommand{\Lcal}{\mathcal{L}}

\newcommand{\Qcal}{\mathcal{Q}}
\newcommand{\Rcal}{\mathcal{R}}

\newcommand{\Mbf}{\mathbf{M}}

\DeclareMathOperator{\id}{id}

\DeclareMathOperator{\rank}{rank}

\DeclareMathOperator{\supp}{supp}
\newcommand{\ee}{\mathrm{e}}

\newcommand{\set}[2]{\left\{\, #1 \ \ \textup{\textbf{:}}\ \ #2 \,\right\}}
\newcommand{\setn}[2]{\{\, #1 \ \ \textup{\textbf{:}}\ \ #2 \,\}}
\newcommand{\setb}[2]{\bigl\{\, #1 \ \ \textup{\textbf{:}}\ \ #2 \,\bigr\}}

\newcommand{\setBB}[2]{\biggl\{\, #1 \ \ \textup{\textbf{:}}\ \ #2 \,\biggr\}}

\newcommand{\norm}[1]{\|#1\|}

\newcommand{\abs}[1]{|#1|}
\newcommand{\abslr}[1]{\left|#1\right|}

\newcommand{\absb}[1]{\bigl|#1\bigr|}

\newcommand{\dd}{\;\mathrm{d}}

\newcommand{\N}{\mathbb{N}}
\newcommand{\R}{\mathbb{R}}

\newcommand{\toweak}{\rightharpoonup}
\newcommand{\toweakstar}{\overset{*}\rightharpoondown}

\newcommand{\sbullet}{\begin{picture}(1,1)(-0.5,-2.5)\circle*{2}\end{picture}}
\newcommand{\frarg}{\,\sbullet\,}

\newcommand{\toY}{\overset{\mathrm{Y}}{\to}}

\newcommand{\term}[1]{\textbf{#1}}

\newcommand{\proofstep}[1]{\textit{#1}}

\def\Xint#1{\mathchoice 
{\XXint\displaystyle\textstyle{#1}}%
{\XXint\textstyle\scriptstyle{#1}}%
{\XXint\scriptstyle\scriptscriptstyle{#1}}%
{\XXint\scriptscriptstyle\scriptscriptstyle{#1}}%
\!\int} 
\def\XXint#1#2#3{{\setbox0=\hbox{$#1{#2#3}{\int}$} 
\vcenter{\hbox{$#2#3$}}\kern-.5\wd0}} 
\def\dashint{\,\Xint-}

\newcommand{\restrict}{\begin{picture}(10,8)\put(2,0){\line(0,1){7}}\put(1.8,0){\line(1,0){7}}\end{picture}}

\renewcommand{\epsilon}{\varepsilon}
\renewcommand{\phi}{\varphi}

\begin{document}


\title[Orientation-preserving Young measures]{Orientation-preserving Young measures}

\author{Konstantinos Koumatos}
\address{\textit{Konstantinos Koumatos:} Mathematical Institute, University of Oxford, Andrew Wiles Building, Radcliffe Observatory Quarter, Woodstock Road, Oxford OX2 6GG, United Kingdom.}
\email{Konstantinos.Koumatos@maths.ox.ac.uk}

\author{Filip Rindler}
\address{\textit{Filip Rindler:} University of Cambridge, Centre for Mathematical Sciences, Wilberforce Road, Cambridge CB3 0WA, United Kingdom.}
\email{F.Rindler@maths.cam.ac.uk}

\author{Emil Wiedemann}
\address{\textit{Emil Wiedemann:} Department of Mathematics, University of British Columbia, and Pacific Institute for the Mathematical Sciences, Vancouver, B.C., Canada V6T 1Z2.}
\email{emil@math.ubc.ca}

\begin{abstract}
We prove a characterization result in the spirit of the Kinderlehrer--Pedregal Theorem for Young measures generated by gradients of Sobolev maps satisfying the orientation-preserving constraint, that is the pointwise Jacobian is positive almost everywhere. The argument to construct the appropriate generating sequences from such Young measures is based on a variant of convex integration in conjunction with an explicit lamination construction in matrix space. Our generating sequence is bounded in $\Lrm^p$ for $p$ less than the space dimension, a regime in which the pointwise Jacobian loses some of its important properties. On the other hand, for $p$ larger than, or equal to, the space dimension the situation necessarily becomes rigid and a construction as presented here cannot succeed. Applications to relaxation of integral functionals, the theory of semiconvex hulls, and approximation of weakly orientation-preserving maps by strictly orientation-preserving ones in Sobolev spaces are given.
\vspace{4pt}



\vspace{4pt}

\noindent\textsc{Date:} \today{} (version 3.0).
\end{abstract}


\hypersetup{
  pdfauthor = {Konstantinos Koumatos (University of Oxford) and Filip Rindler (University of Cambridge) and Emil Wiedemann (University of British Columbia and PIMS)},
  pdftitle = {Orientation-preserving Young measures},
  pdfsubject = {MSC (2010): 49J45 (primary); 28B05, 46G10},
  pdfkeywords = {Gradient Young measure, characterization, convex integration, orientation-preserving deformations, positive Jacobian, laminations}
}


\maketitle




\section{Introduction}

Young measures allow to express limits of certain nonlinear quantities that depend on a weakly converging subsequence, a recurring problem in the Calculus of Variations and the theory of nonlinear PDEs~\cite{Youn37GCEA,Youn69LCVO,Ball89VFTY,Pedr97PMVP,Mull99VMMP}. More specifically, let $(v_j) \subset \Lrm^p(\Omega;\R^N)$ ($\Omega \subset \R^d$ an open set) be a uniformly $\Lrm^p$-bounded sequence (here, $1 \leq p \leq \infty$). Then, the so-called Fundamental Theorem for Young measures assures that there exists a family of probability measures $(\nu_x)_{x \in \Omega}$, indexed by the points from the domain, such that
\[
  \lim_{j\to\infty} \int_\Omega f(x,v_j(x)) \dd x  \;\;\to\;\;  \int_\Omega \int_{\R^N} f(x,A) \dd \nu_x(A) \dd x
\]
for all Carath\'{e}odory functions $f \colon \Omega \times \R^N \to \R$ such that $(f(\frarg,v_j))_j$ is equiintegrable. The family $\nu = (\nu_x)_{x \in \Omega}$ is called the \term{Young measure} generated by the sequence $(v_j)$.

In applications, the sequence $(v_j)$ is usually constrained either by differential or pointwise constraints. Most commonly, \term{gradient Young measures} are considered, i.e.\ those that are generated by a sequence of gradients $(v_j) = (\nabla u_j)$ with $(u_j) \subset \Wrm^{1,p}(\Omega;\R^m)$ uniformly bounded (here, $\R^N = \R^{m \times d}$). For example, in elasticity theory, gradient Young measures have been instrumental in describing the formation of microstructure as a result of non-convex energy minimization \cite{BalJam92PETT,Bhat92SAM,Mull99VMMP}. Immediately, the question arises whether one can see the property of being generated by a sequence of gradients from the Young measure itself. This fundamental problem was solved by the seminal Kinderlehrer--Pedregal Theorem~\cite{KinPed91CYMG,KinPed94GYMG}, which fully characterized gradient Young measures by duality with quasiconvex functions. Various variants and generalizations of the Kinderlehrer--Pedregal result have since emerged in the literature, e.g.~\cite{FoMuPe98ACOE,FonMul99AQLS,FonKru10OCGA,KriRin10CGGY,Rind11?CYMG,SzeWie12YMGI}. In particular, in~\cite{BeKrPa13YMSI} the result was recently extended to Young measures generated by sequences of \emph{invertible} gradients satisfying the uniform bound $\max \{ \abs{\nabla u}, \abs{\nabla u^{-1}} \} \leq \rho$ a.e.\ for some $\rho > 0$.

In physical applications, for sequences $(\nabla u_j) \subset \Lrm^p(\Omega;\R^{d \times d})$, one is often interested in the \emph{pointwise} constraint that the maps underlying the gradients be \term{strictly orientation-preserving}, that is,
\begin{equation} \label{eq:det_nabla_uj_pos}
  \det\, \nabla u_j(x) > 0  \qquad\text{a.e.\ in $\Omega$.}
\end{equation}
For example, in elasticity theory, orientation reversal and interpenetration of matter should be excluded by physical reasoning and hence one requires that admissible deformations in the relevant minimization problem are strictly orientation-preserving and injective almost everywhere. Of course, under regularity assumptions, the positivity of the Jacobian itself relates to (at least local) non-interpenetration of matter; however, for deformations of Sobolev regularity with exponent $p$ below the dimension, the positivity of the Jacobian is not even necessary for injectivity \cite{Hen11SobHomJ0} and this question lies outside the scope of the present work.

Nevertheless, the natural question of characterizing those Young measures that are generated by sequences of gradients of strictly orientation-preserving maps has so far remained open. The reason for the inherent difficulty of this question is the following: Suppose that $(u_j)\subset \Wrm^{1,p}(\Omega,\R^d)$ bounded, $u_j\rightharpoonup u$ in $\Wrm^{1,p}$ and $(\nabla u_j)$ generates the measure $\nu=(\nu_x)_{x\in\Omega}$ so that, in particular, $[\nu]=\nabla u$ a.e. (see (II) below). The proof of the Kinderlehrer--Pedregal Theorem is crucially based on modifying $(u_j)$ to get a new sequence $(v_j)$ such that $(\nabla v_j)$ still generates $\nu$ but $v_j-u\in W^{1,p}_0(\Omega,\R^d)$. This is achieved through standard cut-off techniques which, nevertheless, cannot preserve non-convex constraints such as the orientation-preserving condition. 



So let $(\nabla u_j) \subset \Lrm^p(\Omega;\R^{d \times d})$ generate a Young measure $\nu = (\nu_x)$ and satisfy~\eqref{eq:det_nabla_uj_pos}. Since we are dealing with a sequence of gradients, $\nu$ is a \term{gradient $p$-Young measure}, that is, the usual Kinderlehrer--Pedregal constraints hold:
\begin{itemize}
\item[(I)] $\displaystyle\int_\Omega \int\abs{A}^p \dd \nu_x(A) \dd x <\infty$.
\item[(II)] The barycenter $[\nu](x) := \int A \dd \nu_x(A)$ is a gradient, i.e. there exists $\nabla u \in \Lrm^p(\Omega;\R^{d \times d})$ with $[\nu] = \nabla u$ a.e.
\item[(III)] For every quasiconvex function $h \colon \R^{d \times d} \to \R$ with $\abs{h(A)}\leq c(1+\abs{A}^p)$, the Jensen-type inequality
\[
  \qquad h(\nabla u(x)) \leq  \int h(A) \dd \nu_x(A)  \qquad\text{holds for a.e.\ $x \in \Omega$.}
\]
\end{itemize}
In this context recall that a locally bounded mapping $h \colon \R^{d \times d} \to \R$ is called \term{quasiconvex} if
\[
  h(M) \leq \dashint_{B(0,1)} h(M + \nabla \psi(x)) \dd x
\]
for all $M \in \R^{d \times d}$ and all $\psi \in \Crm_c^\infty(B(0,1);\R^d)$ (compactly supported); the open unit ball $B(0,1)$ can equivalently be replaced by any other open set such that $\abs{\partial \Omega} = 0$. Details about quasiconvex functions can, for example, be found in~\cite{Daco08DMCV}. The function $u$ is called the \term{underlying deformation} of the Young measure.

On the other hand, it is not difficult to verify (see below for a proof) that~\eqref{eq:det_nabla_uj_pos} implies the following pointwise constraint:
\begin{itemize}
\item[(IV)] For a.e.\ $x \in \Omega$,
\[
  \qquad \supp{\nu_x} \in \setb{ M \in \R^{d \times d} }{\det\, M \geq 0}.
\]
\end{itemize}

This paper deals with the question of how, given a Young measure $\nu$ satisfying (I)--(IV), one can recover a sequence $(\nabla u_j) \subset \Lrm^p(\Omega;\R^{d \times d})$ generating $\nu$ such that~\eqref{eq:det_nabla_uj_pos} holds. In particular, we will prove the following main theorem (see below for a discussion on the restrictions on $p$):

\begin{theorem} \label{thm:main}
Let $\Omega \subset \R^d$ be open and bounded such that $|\partial\Omega|=0$, and $p \in (1,d)$. Furthermore, let $\nu = (\nu_x)_{x\in\Omega} \subset \Mbf^1(\R^{d \times d})$ be a $p$-Young measure with underlying deformation $u\in \Wrm^{1,p}(\Omega,\mathbb{R}^d)$. Then the following are equivalent:
\begin{itemize}
  \item[(i)] There exists a sequence of gradients $(\nabla u_j) \subset \Lrm^p(\Omega;\R^{d \times d})$ that generates $\nu$, such that all $\nabla u_j$ are strictly orientation-preserving, that is,
\[
  \qquad \det\, \nabla u_j > 0 \quad \text{a.e.}\qquad\text{for all $j\in\N$}.
\]
  \item[(ii)] The conditions (I)--(IV) hold.
\end{itemize}
Furthermore, if (I)--(IV) hold, the orientation-preserving sequence $(u_j)$ can be chosen so that $u_j-u\in\Wrm^{1,p}_0(\Omega,\R^d)$ and $(\nabla u_j)$ is $p$-equiintegrable.
\end{theorem}

\begin{remark}
Imposing non-trivial constraints on generating sequences from the support of the measure, as in Theorem~\ref{thm:main}, is typical. For example, Zhang's Lemma \cite{Zhang92Lemma} gives $L^{\infty}$ bounds on generating sequences when the support is compact (see also M\"uller \cite{Muller00Zhang} for a refined version of Zhang's Lemma when the support lies in a compact, convex set). For non-convex constraints the situation is more complicated and the first treatment was given by Astala and Faraco \cite{AstFar02QR}. There it is shown that, in two dimensions, gradient Young measures supported on quasiregular sets can be generated by quasiregular mappings. In \cite{AstFar02QR}, as in our result, the quasiregular generating sequence lies in a Sobolev space with restricted exponent (see Section~\ref{sc:remarks} where the constraint $p<d$ is further discussed). We also refer the reader to \cite{BeKrPa13YMSI} (already mentioned) where the authors investigate measures supported on invertible matrices.
\end{remark}

Interestingly, also for weakly orientation-preserving maps (such that only $\det\, \nabla u_j \geq 0$ a.e.) we get the same result. So, as concerning Young measures, these two classes of generating sequences are interchangeable. More specifically, our Theorem~\ref{thm:main} also immediately yields the following corollary, expressing that strictly orientation-preserving deformations are $\Wrm^{1,p}$-dense in the set of weakly orientation-preserving deformations:

\begin{corollary}
Let $\Omega \subset \R^d$ be open and bounded with Lipschitz boundary and $p \in [1,d)$. Let $u \in \Wrm^{1,p}(\Omega;\R^d)$ be weakly orientation-preserving, i.e.
\[
  \det\, \nabla u \geq 0 \quad \text{a.e.}
\]
Then, there exists a sequence $(v_j) \subset \Wrm^{1,p}(\Omega;\R^d)$ that is strictly orientation-preserving, that is
\[
  \det\, \nabla v_j > 0 \quad \text{a.e.}\qquad\text{for all $j \in \N$,}
\]
and such that $\norm{v_j-u}_{1,p} \to 0$ as $j\to\infty$.
\end{corollary}

Note that here we get an approximation in $\Wrm^{1,p}$, i.e.\ for the primitives as well; this follows directly from the Poincar\'{e}--Friedrichs inequality (requiring a regularity for the boundary of $\Omega$) and elementary arguments. Note that, as stated, the above result is also valid for $p=1$.  

A further application of Theorem~\ref{thm:main} is the relaxation of integral functionals in the class of orientation-preserving deformations: 

\begin{theorem}\label{thm:relaxationintro}
Suppose that $\Omega\subset\mathbb{R}^d$ is a bounded Lipschitz domain, $p\in(1,d)$, $\bar{u}\in \Wrm^{1,p}(\Omega,\R^d)$, and let $f: \Omega \times \mathbb{R}^{d\times d}\rightarrow\mathbb{R}$ be a Carath\'{e}odory function satisfying
\[
c(\abs{A}^p-1)\leq f(x,A)\leq C(1+\abs{A}^p)
\]
for all $(x,A)\in\Omega\times\mathbb{R}^{d\times d}$ and constants $0<c\leq C$. Then,
\[
\inf_{\mathcal{A}}\, \int_{\Omega}f(x,\nabla u(x))\dd x=\min_{\mathcal{A}^{YM}}\, \int_{\Omega}\int f(x,A)\dd \nu_x(A)\dd x,
\]
where 
\[
\mathcal{A}:=\setb{u\in \Wrm^{1,p}(\Omega,\mathbb{R}^d)}{u\vert_{\partial\Omega}=\bar{u},\,\det\, \nabla u(x)>0\mbox{ a.e.}},
\]
\[
\mathcal{A}^{YM}:=\setb{\nu\mbox{ gradient $p$-Young measure}}{ \supp \nu_x\subset\{\det\, M\geq 0\},\ [\nu]=\nabla u, u\vert_{\partial\Omega}=\bar{u}}.
\]
\end{theorem}
To the best of the authors' knowledge, relaxation results under the strict orientation-preserving constraint do not exist in the literature and its proof can be found in Section~\ref{sc:relaxation} below (see \cite{AnHMan09RTNE} for a relaxation theorem under the constraint $\det M\neq 0$, $p\in(1,\infty)$). We note that, due to the restriction $p<d$, one cannot expect that $u\in\mathcal{A}$ in the definition of $\mathcal{A}^{YM}$.

Returning to Theorem~\ref{thm:main}, we observe that one direction is straightforward:

\begin{proof}[Proof of \enquote{(i) $\Rightarrow$ (ii)} in Theorem~\ref{thm:main}]
The conditions (I)--(III) follow from standard arguments, this is the easy part of the Kinderlehrer--Pedregal Theorem, see~\cite{KinPed91CYMG,KinPed94GYMG,Pedr97PMVP}. For (IV), take $\phi \in \Crm_c(\Omega)$ and $h \in \Crm_c(\R^{d \times d})$ such that $\supp h \subset\subset \setn{ M \in \R^{d \times d} }{ \det\, M < 0 }$. Then,
\[
  \int_\Omega \phi(x) \int h \dd \nu_x \dd x = \lim_{j\to\infty} \int_\Omega \phi h(\nabla u_j) \dd x = 0
\]
by the assumptions on $\nabla u_j$. Varying $\phi$, we get
\[
  \int h \dd \nu_x = 0  \qquad\text{for a.e.\ $x \in \Omega$.}
\]
Since this holds for all $h$ as above, $\supp \nu_x \subset \setn{ M \in \R^{d \times d} }{ \det\, M \geq 0 }$ for a.e.\ $x \in \Omega$.
\end{proof}

The bulk of this paper is devoted to proving the other implication. For the purpose of illustration, assume for the moment that $p \geq d$. Then, if we take $(\nabla u_j)$ as constructed in the Kinderlehrer--Pedregal Theorem, the characteristic feature of Young measures allows to represent the nonlinear limit of $\det\, \nabla u_j$,
\begin{equation} \label{eq:approx_orientpres}
  \det\, \nabla u_j  \;\;\toweak\;\;  0  \qquad\text{weakly in $\Lrm^1$}.
\end{equation}
This means that the sequence is \enquote{asymptotically orientation-preserving}. It remains to make it \emph{exactly} orientation preserving for every $j \in \N$.

Our strategy is in the spirit of the technique of \term{convex integration} \cite{Grom86PDR,EliMis02IH,MulSve03CILM,Kirc03RGM,AFS08CI&EllipticEqs,DeLSze12HEFD}, but there are some differences. First, we show a result about the \enquote{geometry} of the set $\setn{ M \in \R^{d \times d} }{ \det\, M = 0 }$: Given any matrix $M_0$ with $\det\, M_0<0$, it is always possible to construct an infinite-order $p$-laminate (definition see Section~\ref{sc:geometry}) with barycenter $M_0$ that is supported in the set of matrices with zero determinant. Second, using several iterative steps, we prove a general convergence principle that allows us to improve a generating sequence that is  \enquote{asymptotically  orientation-preserving} in the sense of~\eqref{eq:approx_orientpres} to one that consists only of weakly orientation-preserving gradients, see Section~\ref{sc:conv_weakly_op}. Finally, in Section~\ref{sc:conv_strictly_op} we use yet another iteration to improve this into a sequence of strictly orientation-preserving gradients. All the perturbations and corresponding estimates in our iteration process are obtained \enquote{softly} by repeated use of the fact that laminates are gradient Young measures, rather than by explicit construction (but, of course, the laminates themselves are explicitly constructed in the first step of our proof).  

At this point we remark that the \enquote{classical} convex integration arguments are not directly applicable because laminar oscillations can only give weakly orientation-preserving mappings (note that our condition $\det\, A>0$ defines an \emph{open} set in matrix space). Also, since the $p$-growth condition turns out to be crucial, one could speak of \enquote{$p$-convex integration} for this variant and, further, our convergence principle is different from the one usually employed in convex integration as it involves different generalized convex hulls, see Section~\ref{sc:remarks} for details. This gives rise to an application of convex integration which distinguishes between different degrees of integrability (rather than differentiability), a phenomenon that has, to the authors' knowledge, previously appeared only in~\cite{AFS08CI&EllipticEqs} and work of Yan~\cite{Yan96RemarksStability,Yan01LinearBVP,Yan03Baire} where convex integration techniques are employed for unbounded sets using laminar constructions resembling those of the present work. Indeed, convex integration typically shows flexibility below a certain threshold regularity, whereas at higher regularity the situation is rigid. This is the case e.g. for isometric imbeddings of manifolds~\cite{CoDeSz12HRCI} or incompressible fluid flows~\cite{DeLSze12HEFD}, where convex integration gives flexibility in certain H\"{o}lder spaces. In our case, the threshold integrability is $p=d$ and we show flexibility for $p<d$ and rigidity for $p\geq d$ (cf. Section~\ref{sc:remarks}).

The paper is organized as follows: In the first section we recall preliminary results about Young measures, then in Sections~\ref{sc:geometry}--\ref{sc:conv_strictly_op} we prove the implication \enquote{(ii) $\Rightarrow$ (i)} of Theorem~\ref{thm:main}. Section~\ref{sc:relaxation} is concerned with the proof of Theorem~\ref{thm:relaxationintro}. Finally, we discuss the constraint $p<d$, convex integration, and generalized convex hulls in Section~\ref{sc:remarks}.

\subsection*{Acknowledgments}

The authors wish to thank John Ball, Daniel Faraco, Duvan Henao, Jan Kristensen, Martin Kru\v{z}\'{i}k, and Angkana R\"{u}land for discussions related to the present paper. KK was supported by the European Research Council grant agreement ${\rm n^o}$ 291053.

\section{Gradient Young measures}\label{sc:gradient Young measure}

In all of the following, we use the Frobenius norm, which turns out to be crucial for some estimates. This norm is defined for a matrix $M = (M^i_j) \in \R^{d \times d}$ as follows:
\begin{equation} \label{eq:Frobenius}
  \abs{M} = \abs{M}_F := \left[ \sum_{i,j=1}^d (M^i_j)^2 \right]^{1/2}
  = \left[ \sum_{k=1}^d \sigma_k^2 \right]^{1/2},
\end{equation}
where $\sigma_k$, $k=1,\ldots,d$ are the singular values of $M$, i.e.\ the square roots of the eigenvalues of $M^T M$ or, equivalently, $MM^T$.

Let $\Omega \subset \R^d$ be an open set. A family $(u_j)_{j\in J}\subset \Lrm^p(\Omega)$ is said to be \term{$p$-equiintegrable} if $(|u_j|^p)$ is equiintegrable, i.e. if

\[
\lim_{R\to\infty}\sup_{j\in J}\int_{\{|u_j|^p>R\}}|u_j(x)|^p\dd x=0.
\]

A \term{$p$-Young measure}, $1 \leq p \leq \infty$, is a parametrized family $\nu = (\nu_x)_{x \in \Omega} \subset \Mbf^1(\R^N)$ of probability measures on $\R^N$ (which are collected in the set $\Mbf^1(\R^N)$) such that the following conditions are satisfied:
\begin{enumerate}
\item[(1)] The family $(\nu_x)$ is \term{weakly* measurable}, that is, for every Borel set $B \subset \R^N$ the map $x \mapsto \nu_x(B)$ is ($\Lcal^d \restrict \Omega$)-measurable.
\item[(2)] The map $x \mapsto \int \abs{A}^p \dd \nu_x$ lies in $\Lrm^1(\Omega)$.
\end{enumerate}
We say that a sequence $(u_j)$, bounded in $L^p(\Omega;\R^N)$, \term{generates} the Young measure $\nu$ if 
\[
  \int_\Omega f(x,u_j(x)) \dd x  \;\;\to\;\;
  \int_\Omega \int f(x,A) \dd \nu_x(A) \dd x 
  \]
for every Carath\'{e}odory function $f:\Omega\times\R^N\to \R$ (i.e. every function which is measurable in the first and continuous in the second argument) such that $(f(\frarg,u_j))$ is equiintegrable. In this case we write $u_j\toY \nu$.

We have the following lemma, which expresses a well-known fact:
\begin{lemma}\label{lem:pclosegeneration}
Suppose $(u_j)$ and $(v_j)$ are $\Lrm^p(\Omega)$-bounded sequences, $(u_j)$ generates the Young measure $\nu$ and 
\[
\lim_{j\to\infty}\norm{u_j-v_j}_p=0.
\]
Then also $(v_j)$ generates $\nu$.
\end{lemma}

We denote the \term{barycenter} of a $p$-Young measure $\nu$ by
\[
  [\nu](x) := \int A \dd \nu_x(A),  \qquad x \in \Omega,
\]
and $[\nu]$ lies in $\Lrm^p(\Omega;\R^N)$ by condition~(2) above. A Young measure $\nu$ is called \term{homogeneous} if $x \mapsto \nu_x$ is an almost everywhere constant map, i.e.\ $\nu_x = \nu \in \Mbf^1(\R^N)$ for a.e.\ $x \in \Omega$. Many properties of Young measures are collected in~\cite{Pedr97PMVP}; for example, it can be shown that all $p$-Young measures according to the above definition are generated by some sequence of uniformly $\Lrm^p$-bounded functions with values in $\R^N$.

We say that a sequence of $p$-Young measures $\nu^j$ \term{converges weakly*} to a $p$-Young measure $\nu$, in symbols $\nu^j \toweakstar \nu$ or $\nu^j\toY \nu$, if
\[
  \int_\Omega \int f(x,A) \dd \nu^j_x(A) \dd x  \;\;\to\;\;
  \int_\Omega \int f(x,A) \dd \nu_x(A) \dd x 
  \qquad\text{for all $f \in \Crm_c^\infty(\Omega \times \R^N)$.}
\]
Clearly, for homogeneous Young measures, the weak* convergences in the sense of Young measures and in the sense of (ordinary) measures coincide.

A very important subclass of Young measures is the class of those that can be generated by sequences of \emph{gradients}, the so-called \term{gradient Young measures} (in this work we will only be concerned with maps $u \colon \Omega \subset \R^d \to \R^d$, whereby for our Young measures $\R^N = \R^{d \times d}$). The fundamental result in this context is the Kinderlehrer--Pedregal Theorem~\cite{KinPed91CYMG,KinPed94GYMG} already mentioned in the introduction: A $p$-Young measure is generated by a sequence of uniformly $\Lrm^p$-bounded gradients if and only if the conditions (I)--(III) from the introduction are satisfied. We call such measures \term{gradient $p$-Young measures}.

The following lemma, which will be used at various steps in the proof of Theorem \ref{thm:main}, is an easy consequence of the proof of this characterization of gradient $p$-Young measures:

\begin{lemma}\label{lem:KinPedrefined} 
Let $\Omega$ be a bounded domain with $\vert\partial\Omega\vert=0$ and let $\left(\nu_x\right)_{x\in\Omega}$ be a gradient $p$-Young measure, $p\in (1,\infty)$, with barycenter $[\nu] = \nabla u$ a.e.\ for some $u \in \Wrm^{1,p}(\Omega;\R^d)$. Set
\[
  S = \setb{x\in\Omega}{\nu_x=\delta_{\nabla u(x)}}.
\]
Then, there exists a generating sequence $(\nabla u_j)\subset (\Lrm^p \cap \Crm^{\infty})(\Omega;\R^{d \times d})$ for $\nu$ which is $p$-equiintegrable and such that
\begin{equation}\label{eq:almostequal}
\lim_{j\rightarrow\infty}\, \absb{\setb{x\in S}{\nabla u_j(x)\neq\nabla u(x)}} =0.
\end{equation}
In addition, $(u_j) \subset \Wrm^{1,p}(\Omega;\R^d)$ can be chosen so as to also satisfy $u_j-u\in\Wrm^{1,p}_0(\Omega,\R^d)$.
\end{lemma}

\begin{proof}
By a standard shifting argument, we may assume that $\nabla u\equiv 0$ and, without loss of generality, that $u\equiv 0$. Fix $j\in\mathbb{N}$. Since the space $\Mbf^1(\R^{d \times d})$ of probability measures is compact and metrizable with respect to the weak* topology, we may cover it with finitely many weakly* closed balls $B^j_i \subset \Mbf(\R^{d \times d})$, $i=1,\ldots, N_j$, of radius $1/(2j)$. We can view $\nu$ as a measurable map from $\Omega$ into the set $\Mbf^1(\R^{d \times d})$ and hence we may define measurable subsets of $\Omega$ by $\tilde{E}^j_i:=\nu^{-1}(B^j_i)$. In particular, we may enumerate the sets $B^j_i$ in such a way that $\delta_0\in\tilde{E}^j_1$ for all $j\in\mathbb{N}$. Note that then $S\subset\tilde{E}^j_1$.

Next, define a family of disjoint measurable sets $E^j_1,\ldots, E^j_{N_j}$ by setting $E^j_1=\tilde{E}^j_1$ and $E^j_i=\tilde{E}^j_i\setminus (E^j_1\cup\ldots\cup E^j_{i-1})$ for $i\geq 2$. Let
\[
F^j:=\bigcup^{N_j}_{i=2}E^j_i.
\]
We may assume (up to a subsequence in $j$) that $\vert F^j\vert >0$, as otherwise, $\vert F^j\vert =0$ for all $j$ implies that $\Omega = S\cup N$, where $N$ is a null set and the constant sequence $\nabla u_j\equiv 0$ would suffice to prove our result.

Then we may find compact sets $K^j_1\subset E^j_1$ and $K^j_2\subset F^j$ (if $|E^j_1|=0$, set $K^j_1=\emptyset$) such that
\begin{equation}\label{eq:compactapprox}
\vert E^j_1\setminus K^j_1\vert + \vert F^j\setminus K^j_2\vert < \frac{1}{j}.
\end{equation}
Finally, since the distance between $K^j_1$ and $K^j_2$ is positive, we may choose disjoint open sets $U^j_1$ and $U^j_2$ such that $K^j_i\subset U^j_i$ and $\vert\partial U^j_i \vert =0$ for $i=1,2$. For a.e.~$x\in\Omega$, define the measures
\[
\nu ^j_x= \begin{cases} \nu_x, & \text{if $x\in U^j_2$}\\
  \delta_0, &\text{otherwise.}
\end{cases}
\]
For each $j\in\N$, $\nu^j$ is a gradient $p$-Young measure and it is readily seen that $\nu^j \toY \nu$ (cf. Proposition 4.24 in~\cite{Mull99VMMP}).

Assume that the original gradient $p$-Young measure $\nu$ is generated by a $p$-equiintegrable sequence $(\nabla v_k) \subset (\Lrm^p \cap \Crm^\infty)(\Omega;\R^{d \times d})$; note that it is always possible to find a $p$-equiintegrable generating sequence by a suitable cut-off argument, see for example Lemma~8.15 in~\cite{Pedr97PMVP}. We aim to find an explicit generating sequence for $\nu^j$ in terms of $\nabla v_k$ so that we can get good estimates for equiintegrability. To this end, we wish to fix the boundary values of $v_k$ on $\partial U^j_2$ to be $0$. We follow a standard cut-off argument but we write it explicitly with a view towards the estimates. For each $j\in\mathbb{N}$ fixed, define a sequence of cut-off functions $\{\eta^j_n\}$ with the following properties:
\begin{itemize}
\item [(i)] $\eta^j_n\equiv 1$ on $\partial U^j_2$;
\item [(ii)] $\eta^j_n\equiv 0$ in $\{x\in U^j_2 : {\rm dist}(x,\partial U^j_2)\geq 1/n\}$;
\item [(iii)] $\vert\nabla\eta^j_n\vert\leq Cn$, for some constant $C > 1$.
\end{itemize}
Consider the functions
\[
w^j_{k,n}(x)=(1-\eta^j_n(x))v_k(x).
\]
Then the $w^j_{k,n}$ satisfy the zero boundary condition on $\partial U^j_2$ for all $k, n$ and also
\[
\nabla w^j_{k,n}=(1-\eta^j_n(x))\nabla v_k -v_k\otimes\nabla\eta^j_n.
\]
Since $v_k\rightarrow 0$ strongly in $\Lrm^p$, we can choose $k=k(n)$, with $k(n)\to\infty$ as $n\to\infty$, such that
\[
\lim_{n\to\infty}\norm{v_{k(n)}\otimes\nabla\eta^j_n}_p=0
\]
uniformly in $j$ (since $\eta^j_n$ is bounded independently of $j$). Moreover, it is easy to see that, for every $j$,
\[
(1-\eta^j_n)\nabla v_{k(n)} \toY \nu^j
\]
as $n\to\infty$. Putting both these properties together we conclude 
\[
\nabla w^j_{k(n),n}(x) \toY \nu^j
\]
for every $j$. But then, in view of $\nu^j\toY \nu$, we can use a standard diagonal argument to choose $n=n(j)$, $n(j)\to\infty$ as $j\to\infty$, so large that
\[
\nabla u_j:=\nabla w^j_{k(n(j)),n(j)} \toY \nu
\]
as $j\to\infty$. By construction, each member of this sequence is compactly supported in $U^j_2$ and therefore in particular in $\Omega$. Also, the $\nabla u_j$ are zero in $U^j_1\supset K^j_1$, which by~\eqref{eq:compactapprox} implies~\eqref{eq:almostequal}. It remains to check the $p$-equiintegrability.

To this end, recall that
\[
\nabla u_j(x)=(1-\eta^j_{n(j)}(x))\nabla v_{k(n(j))} -v_{k(n(j))}\otimes\nabla\eta^j_{n(j)}.
\]
The first term is $p$-equiintegrable because $(\nabla v_k)$ is, and because $|1-\eta^j_{n(j)}(x)|$ is pointwise dominated by 1. The second term converges to zero in $\Lrm^p$ by choice of $k(n)$ (and by choosing $n=n(j)$ even larger if necessary) and is in particular $p$-equiintegrable. This shows that $(\nabla u_j)$ is $p$-equiintegrable, which completes the proof.
\end{proof}

We finish this section with definitions relating to a fundamental subclass of gradient Young measures, the laminates. In this context, see Chapter~9 of~\cite{Pedr97PMVP}.

\begin{definition}
We define:
\begin{enumerate}
\item A collection $\{(t_k,M_k)\}_{k=1,\ldots,m} \subset (0,1) \times \R^{m \times d}$ with $\sum t_k = 1$ is said to satisfy the \term{$(H_m)$-condition}
\begin{itemize}
\item[(i)] for $m = 2$, if $\rank (M_1-M_2) \leq 1$,
\item[(ii)] for $m > 2$, if after a permutation of indices, $\rank (M_1-M_2) \leq 1$ and with
\begin{align*}
  \qquad\qquad s_1 &:= t_1 + t_2,  & \tilde{M}_1 &:= \frac{t_1}{s_1} M_1 + \frac{t_2}{s_1} M_2, \\
  \qquad\qquad s_k &:= t_{k+1},  & \tilde{M}_k &:= M_{k+1} \qquad\text{for}\quad k = 2,3,\ldots,
\end{align*}
the collection $(s_k,\tilde{M}_k)_{k=1,\ldots,m-1}$ satisfies the $(H_{m-1})$-condition.
\end{itemize}

\item A probability measure $\nu \in \Mbf^1(\R^{m \times d})$ is called a \term{finite-order laminate} if $\nu = \sum_{k=1}^m t_k \delta_{M_k}$ and the collection $\{(t_k,M_k)\}_{k=1,\ldots,m} \subset (0,1) \times \R^{m \times d}$ satisfies the $(H_m)$-condition.

\item A probability measure $\nu \in \Mbf^1(\R^{m \times d})$ is called a \term{$p$-laminate} ($1 \leq p < \infty$) if there exists a sequence $(\nu_j) \subset \Mbf^1(\R^{m \times d})$ of finite-order laminates such that $\int \abs{\frarg}^p \dd \nu_j \leq C$ for some constant $C>0$ and all $j \in \N$ and $\nu_j \toweakstar \nu$.

\item A probability measure $\nu$ is called an \term{$\infty$-laminate} if there exists a sequence $(\nu_j) \subset \Mbf^1(\R^{m \times d})$ of finite-order laminates such that $\supp \nu_j \subset K$ for some $K\subset \R^{m \times d}$ compact and all $j \in \N$ and $\nu_j \toweakstar \nu$.
\end{enumerate}
\end{definition}

\begin{remark}
Any finite-order laminate is an $\infty$-laminate and every $\infty$-laminate is a $p$-laminate for every $p\in[1,\infty)$. Also, since finite-order laminates are gradient $\infty$-Young measures (see Chapter~9~\cite{Pedr97PMVP}), by a standard diagonal argument, $\infty$-laminates are gradient $\infty$-Young measures and $p$-laminates are gradient $p$-Young measures.
\end{remark}

\section{Geometry of the determinant constraint} \label{sc:geometry}

In this section we investigate the \enquote{geometry} of the set $\setn{ M \in \R^{d \times d} }{ \det\, M = 0 }$, which has a central place in our argument. First, we make the simple observation that any square matrix $M_0 \in \R^{d \times d}$ with $\det\, M_0 < 0$ can be written as the barycenter of a probability measure $\mu \in \Mbf^1(\R^{d \times d})$ with
\[
  \supp \mu \subset \setb{ M \in \R^{d \times d} }{ \det\, M = 0 }.
\]
Indeed, if (and we will see in the proof of Proposition~\ref{prop:geometry} below that we can always reduce to this case)
\[
  M_0 = \begin{pmatrix} -\sigma_1 & & & \\ & \sigma_2 & & \\ & & \ddots & \\ & & & \sigma_d \end{pmatrix}
  \qquad\text{with}\qquad
  0\leq\sigma_1 \leq \sigma_2 \leq \cdots \leq \sigma_d,
\]
then trivially,
\begin{align*}
  M_0 &=\phantom{:} \frac{1}{2} \begin{pmatrix} 0 & & & & \\ & 2\sigma_2 & & & \\ & & \sigma_3 & & \\ & & & \ddots & \\ & & & & \sigma_d \end{pmatrix}
    + \frac{1}{2} \begin{pmatrix} -2\sigma_1 & & & & \\ & 0 & & & \\ & & \sigma_3 & & \\ & & & \ddots & \\ & & & & \sigma_d \end{pmatrix}  \\
  &=: \frac{1}{2} M_1 + \frac{1}{2} M_2.
\end{align*}
It is clear that $\det\, M_1 = \det\, M_2 = 0$, and so,
\[
  \mu := \frac{1}{2} \delta_{M_1} + \frac{1}{2} \delta_{M_2}
\]
fulfills the above assertion.

A more intricate question is whether this can also be achieved if $\mu$ is restricted to be a gradient Young measure or even a $p$-laminate. This question as well turns out to have a positive answer: It is indeed always possible to write $M_0$ as the barycenter of a $p$-laminate, albeit one with infinite order, and certain good estimates hold. This can be seen as an assertion about the \enquote{geometry} of the set $\setn{ M \in \R^{d \times d} }{ \det\, M = 0 }$, see Section~\ref{sc:remarks} for further discussion of this point.

\begin{proposition} \label{prop:geometry}
Let $M_0 \in \R^{d \times d}$ with $\det\, M_0 < 0$. Then, there exists a homogeneous Young measure $\nu \in \Mbf^1(\R^{d \times d})$ that is a $p$-laminate of infinite order for every $p \in [1,d)$ and such that the following assertions hold:
\begin{itemize}
  \item[(i)] $[\nu] = \displaystyle \int \id \dd \nu = M_0$,
  \item[(ii)] $\supp \nu \subset \setb{ M \in \R^{d \times d} }{ \det\, M = 0 }$,
  \item[(iii)] $\displaystyle \int \abs{\frarg}^p \dd \nu \leq C_p \abs{M_0}^p$,
  \item[(iv)] $\displaystyle \int \abs{A-M_0}^p \dd \nu(A) \leq C_p \abs{\det\, M_0}^{p/d}$,
 \end{itemize}
where $C_p = C(d,p)$.
\end{proposition}

\begin{remark}
\begin{enumerate}
\item Note that $\nu$ does not depend on $p$.
\item We remark that in~(iii),~(iv) and below $C_p=C(d,p)$ denotes a generic constant, which may change from line to line, such that $C_p \to \infty$ as $p \to d$; for the necessity of $p < d$, see Section~\ref{sc:remarks}.
\item One can also show the additional estimate $\displaystyle \int \abs{\det\, A}^{p/d} \dd \nu(A) \leq C_p \abs{\det\, M_0}^{p/d}$.
\end{enumerate}
\end{remark}

\begin{proof}
The idea of the proof is to employ recursive lamination constructions to furnish a sequence of homogeneous Young measures $\nu_0 = \delta_{M_0}, \nu_1, \nu_2, \ldots$, which push more and more of the total mass into the set of zero-determinant matrices, and then use weak*-precompactness of the sequence $(\nu_j)$ to pass to an infinite-order $p$-laminate $\nu$, which satisfies all the properties in the proposition.

\proofstep{Step~1.}
We first transform $M_0$ to diagonal form. Let $M_0 = \tilde{P} \tilde{D}_0 \tilde{Q}^T$ be the \emph{real} singular value decomposition, that is, $\tilde{D}_0 = \mathrm{diag}(\sigma_1,\ldots,\sigma_d)$ with $0<\sigma_1 \leq \sigma_2 \leq \cdots \leq \sigma_d$, and $\tilde{P},\tilde{Q}$ orthogonal matrices. As $0 > \det\, M_0 = \det\, \tilde{P} \cdot \det\, \tilde{D}_0 \cdot \det\, \tilde{Q}$, either $\tilde{P}$ or $\tilde{Q}$ has negative determinant, say $\det\, \tilde{P} < 0$ (the other case is similar). With
\[
  D_0 := \begin{pmatrix} -\sigma_1 & & & \\ & \sigma_2 & & \\ & & \ddots & \\ & & & \sigma_d \end{pmatrix}  \qquad
  P := \tilde{P} \cdot \begin{pmatrix} -1 & & & \\ & 1 & & \\ & & \ddots & \\ & & & 1 \end{pmatrix},  \qquad
  Q := \tilde{Q},
\]
we have $M_0 = P D_0 Q^T$, where now $P,Q \in \mathrm{SO}(d)$ and $\det\, D_0 < 0$. Now, if $D_0$ can be written as a laminate, i.e.\ a hierarchical decomposition along rank-one lines, then the same holds true for $M_0$ since $P (a \otimes b) Q^T = (Pa) \otimes (Qb)$ for any $a,b \in \R^d$. 

We remark in this context that the procedure to reduce to a diagonal matrix does not change the (Frobenius) matrix norm, since the latter only depends on the singular values, which trivially are not changed by the singular value decomposition. Also, as $P,Q \in \mathrm{SO}(d)$, the determinant is also not changed in this process.

\proofstep{Step~2.}
Owing to Step~1, in the following we can assume that $M_0$ is already diagonal, the first diagonal entry is negative and all others are positive. We will write the first $2 \times 2$ block of $M_0$ as an infinite hierarchy of convex combinations along rank-one lines such that all resulting matrices have zero determinant. Write
\[
 M_0 = \begin{pmatrix} -\sigma_1 & & & \\ & \sigma_2 & & \\ & & \ddots & \\ & & & \sigma_d \end{pmatrix},
\]
for which $\sigma_i > 0$ as in Step~1.

Set $r:=2^{\frac{p}{d}-1}$ and observe that since $p < d$, we have $2^{(1-d)/d}\leq r<1$. We also set
$\gamma := \sqrt{\sigma_1\sigma_2}$.
Then, we can decompose $M_0$ twice along rank-one lines as follows:
\begin{align*}
  M_0 &= \frac{1}{2} \bigl[ M_0 + \gamma (\ee_1 \otimes \ee_2) \bigr] + \frac{1}{2} \bigl[ M_0 - \gamma (\ee_1 \otimes \ee_2) \bigr] \\
  &= \frac{1}{4} \bigl[ M_0 + \gamma (\ee_1 \otimes \ee_2) + \gamma (\ee_2 \otimes \ee_1) \bigr]
    + \frac{1}{4} \bigl[ M_0 + \gamma (\ee_1 \otimes \ee_2) - \gamma (\ee_2 \otimes \ee_1) \bigr] \\
  &\qquad + \frac{1}{4} \bigl[ M_0 - \gamma (\ee_1 \otimes \ee_2) + \gamma (\ee_2 \otimes \ee_1) \bigr]
    + \frac{1}{4} \bigl[ M_0 - \gamma (\ee_1 \otimes \ee_2) - \gamma (\ee_2 \otimes \ee_1) \bigr] \\
  &=: \frac{1}{4} M_{1,B1} + \frac{1}{4} M_{1,G1} + \frac{1}{4} M_{1,G2} + \frac{1}{4} M_{1,B2}.
\end{align*}
We can compute
\begin{align}
  \det\, M_{1,G1} &= \det\, M_{1,G2} = (-\sigma_1\sigma_2 + \sigma_1\sigma_2)\prod_{i=3}^d\sigma_i  =  0.  \notag \\ 
  \det\, M_{1,B1} &= \det\, M_{1,B2} = (-\sigma_1\sigma_2 - \sigma_1\sigma_2)\prod_{i=3}^d\sigma_i  = -2\sigma_1\sigma_2\prod_{i=3}^d\sigma_i < 0 \notag\\
  \abslr{\det\, M_{1,B1}} &= \abslr{\det\, M_{1,B2}} =2\abs{\det\, M_0}= (2r)^{d/p} \, \abs{\det\, M_0}.  \label{eq:det_bad_est}
\end{align}
Thus, the \enquote{good} matrices $M_{1,G1}, M_{1,G2}$ already satisfy our constraint of having zero determinant, the \enquote{bad} matrices $M_{1,B1}, M_{1,B2}$ will be further decomposed later on.
Moreover, note that
\begin{equation}  \label{eq:M_1X_dist}
  \abs{M_{1,J}-M_0} =  2^{1/2} (\sigma_1\sigma_2)^{1/2}
    \leq 2^{1/2} \, \abs{\det\, M_0}^{1/d},
\end{equation}
since $0<\sigma_1 \leq \sigma_2 \leq \cdots \leq \sigma_d$ and hence $(\sigma_1\sigma_2)^{d/2}\leq \abs{\det\, M_0}$.

\proofstep{Step~3.}
Define
\[
  \nu_0 := \delta_{M_0},  \qquad
  \nu_1 := \frac{1}{4} \delta_{M_{1,G1}} + \frac{1}{4} \delta_{M_{1,G2}} + \frac{1}{4} \delta_{M_{1,B1}} + \frac{1}{4} \delta_{M_{1,B2}},
\]
and, as detailed above, we observe that $\nu_1$ is derived from $\nu_0$ by two additional lamination steps. Moreover, $[\nu_1] = [\nu_0] = M_0$.

Now recursively apply the procedure from the preceding steps to decompose the \enquote{bad} matrices $M_{1,B1}$ and $M_{1,B2}$ in turn taking the role of $M_0$. This yields matrices $M_{2,G1},\ldots,M_{2,G4}$, $M_{2,B1},\ldots,M_{2,B4}$ such that
\begin{align*}
  M_{1,B1} = \frac{1}{4} M_{2,G1} + \frac{1}{4} M_{2,G2} + \frac{1}{4} M_{2,B1} + \frac{1}{4} M_{2,B2}, \\
  M_{1,B2} = \frac{1}{4} M_{2,G3} + \frac{1}{4} M_{2,G4} + \frac{1}{4} M_{2,B3} + \frac{1}{4} M_{2,B4}.
\end{align*}
We define $\nu_2$ accordingly as
\begin{align*}
  \nu_2 &:= \frac{1}{4} \delta_{M_{1,G1}} + \frac{1}{4} \delta_{M_{1,G2}} + \frac{1}{4^2} \biggl[ \delta_{M_{2,G1}} + \delta_{M_{2,G2}} + \delta_{M_{2,B1}} + \delta_{M_{2,B2}} \biggr] \\
  &\qquad + \frac{1}{4^2} \biggl[ \delta_{M_{2,G3}} + \delta_{M_{2,G4}} + \delta_{M_{2,B3}} + \delta_{M_{2,B4}} \biggr].
\end{align*}
Then, still $[\nu_2] = M_0$ and $\nu_2$ is a finite-order laminate.

Now iterate this scheme of first bringing the matrix to diagonal form via Step~1 and then laminating via Step~2, in every step defining a new finite-order laminate $\nu_j$, $j \in \N$, with $[\nu_j] = M_0$. In this context recall that the reduction to a diagonal form does not change the matrix norm or determinant.

In more detail, we get in the first two iterations (adding appropriate indices to the matrices $P,Q,D$):
\begin{align*}
  M_0 &= P_0 D_0 Q_0^T \\
  &= P_0 \biggl( \frac{1}{4} M_{1,G1} + \frac{1}{4} M_{1,G2} + \frac{1}{4} M_{1,B1} + \frac{1}{4} M_{1,B2} \biggr) Q_0^T \\
  &= P_0 \biggl( \frac{1}{4} M_{1,G1} + \frac{1}{4} M_{1,G2} + \frac{1}{4} P_{1,B1} D_{1,B1} Q_{1,B1}^T + \frac{1}{4} P_{1,B2} D_{1,B2} Q_{1,B2}^T \biggr) Q_0^T \\
  &= \frac{1}{4} \underbrace{P_0 M_{1,G1} Q_0^T}_{\det\, = 0} + \frac{1}{4} \underbrace{P_0 M_{1,G2} Q_0^T}_{\det\, = 0} + \frac{1}{4} P_0 P_{1,B1} D_{1,B1} Q_{1,B1}^T Q_0^T + \frac{1}{4} P_0 P_{1,B2} D_{1,B2} Q_{1,B2}^T Q_0^T \\
  &= \frac{1}{4} P_0 M_{1,G1} Q_0^T + \frac{1}{4} P_0 M_{1,G2} Q_0^T \\
  &\qquad + \frac{1}{4} P_0 P_{1,B1} \biggl( \frac{1}{4} M_{2,G1} + \frac{1}{4} M_{2,G2} + \frac{1}{4} M_{2,B1} + \frac{1}{4} M_{2,B2} \biggr) Q_{1,B1}^T Q_0^T + \underbrace{\cdots}_{\text{$1,B2$-part}}
\end{align*}
In every step of bringing matrices to diagonal form, the mean value $M_0$ of the Young measures $\nu_j$ associated to these splittings is preserved. Further, note that we only split along rank-one lines, hence
\[
  P_0 M_{1,G1/G2/B1/B2} Q_0^T = M_0 \pm \gamma (P_0\ee_1) \otimes (Q_0\ee_2) \pm \gamma (P_0\ee_2) \otimes (Q_0\ee_1),
\]
and we preserve the property for the $\nu_j$'s to be finite-order laminates.

\proofstep{Step~4.}
Let us consider the distance integral in~(iv):
\begin{align*}
  \int \abs{A-M_0}^p \dd \nu_j(A) &= \sum_{i=1}^j \sum_{k=1}^{2^i} \frac{1}{4^i} \abs{M_{i,Gk}-M_0}^p + \sum_{k=1}^{2^j} \frac{1}{4^j} \abs{M_{j,Bk}-M_0}^p \\
  &\leq \sum_{i=1}^j \sum_{k=1}^{2^i} \frac{1}{4^i} \biggl( \sum_{\ell=1}^i \abs{X_\ell -X_{\ell-1}} \biggr)^p + \sum_{k=1}^{2^j} \frac{1}{4^j} \biggl( \sum_{\ell=1}^j \abs{Y_\ell-Y_{\ell-1}} \biggr)^p,
\end{align*}
where in the innermost summations we defined $X_i := M_{i,Gk}$, $X_0 := M_0$, and $X_{\ell-1}$ is the $M_{\ell-1,Bk}$ with $k \in \{1,\ldots,2^{\ell-1}\}$ such that $X_\ell$ originated from $X_{\ell-1}$ through the lamination construction from the previous proof step (with the understanding $M_{0,B1} := M_0$); similarly, $Y_j := M_{j,Bk}$, $Y_0 := M_0$, and $Y_{\ell-1}$ defined analogously to $X_{\ell-1}$. Then, $\sum_{\ell=1}^i X_\ell - X_{\ell-1} = M_{i,Gk}-M_0$ and $\sum_{\ell=1}^j Y_\ell - Y_{\ell-1} = M_{j,Bk}-M_0$, and so the second line in the estimate follows from the first by virtue of the triangle inequality. Now, to bound $\abs{X_\ell - X_{\ell-1}}$ we use~\eqref{eq:M_1X_dist} and then~\eqref{eq:det_bad_est} recursively. Thus,
\begin{align*}
  \sum_{\ell=1}^i \abs{X_\ell -X_{\ell-1}} &\leq \sum_{\ell=1}^i 2^{1/2} \, \abs{\det\, X_{\ell-1}}^{1/d} \leq \sum_{\ell=1}^i 2^{1/2} \cdot (2r)^{(\ell-1)/p} \, \abs{\det\, M_0}^{1/d} \\
  &\leq \frac{2^{1/2} \, \abs{\det\, M_0}^{1/d}}{(2r)^{1/p}-1} \cdot (2r)^{i/p}
\end{align*}
and a similar estimate holds for the second inner summation involving the $Y_\ell$'s. Hence, we can plug this into the previous estimate to get 
\begin{align}
 \int \abs{A-M_0}^p \dd \nu_j(A) &\leq \biggl[\frac{2^{1/2}}{(2r)^{1/p}-1} \biggr]^p \cdot  \abs{\det\, M_0}^{p/d} \cdot \left[ \sum_{i=1}^j \frac{2^i (2r)^i}{4^i} + \frac{2^j (2r)^j}{4^j} \right]\nonumber \\
  &\leq \biggl[\frac{2^{1/2}}{(2r)^{1/p}-1} \biggr]^p \cdot  \abs{\det\, M_0}^{p/d} \cdot \biggl[ \frac{1}{1-r} + r^j \biggr] \nonumber\\
  &\leq C_p \abs{\det\, M_0}^{p/d}.
  \label{eq:geom_prop1}
\end{align}
Moreover, by \eqref{eq:geom_prop1} and the fact that the $\nu_j$'s are probability measures,
\begin{align}
 \int \abs{A}^p \dd \nu_j(A) & \leq 2^p \left[\int \abs{A-M_0}^p \dd \nu_j(A) + \abs{M_0}^p \right]\nonumber\\
  & \leq 2^p C_p \abs{\det\, M_0}^{p/d} + 2^p \abs{M_0}^p\nonumber\\
  & \leq C_p \abs{M_0}^p,
  \label{eq:geom_prop2}
\end{align}
which is uniformly bounded. In particular, the $\nu_j$ are (sequentially) weakly*-precompact as measures, hence there exists a subsequence and a cluster point $\nu \in \Mbf(\R^{d \times d})$, which is a $p$-laminate, $p \in [1,d)$, and satisfies $[\nu] = M_0$. 
Passing to the limit in \eqref{eq:geom_prop1} and \eqref{eq:geom_prop2} yields~(iii) and~(iv).

Finally, it can be seen easily that the mass of $\nu_j$ that is carried by \enquote{bad} matrices, i.e.\ those with negative determinant, is
\[
  \abs{\nu_j} \bigl( \setb{ M \in \R^{d \times d} }{ \det\, M < 0 } \bigr) = \frac{2^j}{4^j} \to 0  \qquad\text{as $j\to\infty$.} 
\]
Thus, also~(ii) follows, concluding the proof.
\end{proof}

\begin{remark}
By a similar, slightly more intricate, strategy one can also show that there exist (finite-order) laminates $\nu_j$, with $\int \abs{\frarg}^p\dd \nu_j$ uniformly bounded, and $\nu_j$ can be split as
\[
  \nu_j = \nu_j^+ + \nu_j^- \qquad\text{with}\qquad
  \supp \nu_j^\pm \subset \setb{ M \in \R^{d \times d} }{ \det\, M \gtrless 0 },
\]
where $\int \abs{\frarg}^p \dd \nu_j^- \to 0$ as $j\to\infty$. In particular, $\nu_j\toweakstar \nu$ (in the weak* Young measure or measure convergence) where $\nu$ is as in Proposition~\ref{prop:geometry} but $\supp \nu \subset \setb{ M \in \R^{d \times d} }{ \det\, M > 0 }$.
\end{remark}

\section{Weakly orientation-preserving generating sequences} \label{sc:conv_weakly_op}

Employing our investigation into the geometry of the zero-determinant constraint in matrix space from the previous section and the fact that $p$-laminates are gradient Young measures (which follows for example from the Kinderlehrer--Pedregal Theorem), in this section we prove the following proposition, which directly entails a weaker variant of Theorem~\ref{thm:main} with the generating sequence consisting of gradients with nonnegative determinant only; the full strength of the main theorem is proved in the following section.

\begin{proposition}\label{prop:nonnegapprox}
Let $u\in \Wrm^{1,p}(\Omega;\R^d)$, $p\in (1,d)$ and $\epsilon>0$. Then there exists $v\in \Wrm^{1,p}(\Omega;\R^d)$ such that
\begin{itemize}
\item[(i)]$\det\, \nabla v(x)\geq0$ for a.e.\ $x \in \Omega$,\vspace{0.2cm}
\item[(ii)] $v-u\in\Wrm^{1,p}_0(\Omega,\R^d)$,\vspace{0.2cm}
\item[(iii)]$\displaystyle \norm{\nabla u-\nabla v}^p_p\leq C_p\int_{\{\det\, \nabla u<0\}}|\det\, \nabla u(x)|^{p/d}\dd x,$\vspace{0.2cm}
\item[(iv)]$\absb{\setb{x\in\Omega}{ \text{$\det\, \nabla u\geq0$ and $\nabla v(x)\neq\nabla u(x)$}}}<\epsilon,$
\end{itemize}
where $C_p=C(d,p)$.
\end{proposition}
Before we prove the proposition, let us demonstrate how the weaker version of Theorem~\ref{thm:main} follows from it.

\begin{proof}[Proof of the weaker version of \enquote{(ii) $\Rightarrow$ (i)} in Theorem~\ref{thm:main}]

For $p\in (1,d)$, let a gradient $p$-Young measure $\nu$ be given such that $\supp \nu_x \subset \setb{ M \in \R^{d \times d} }{ \det\, M \geq 0 }$ for a.e.\ $x \in \Omega$. By Lemma~\ref{lem:KinPedrefined}, there exists a generating sequence $(\nabla u_j)$ for $\nu$ which is $p$-equiintegrable and satisfies $u_j-u\in\Wrm^{1,p}_0(\Omega,\R^d)$ on where $\nabla u=[\nu]$. Moreover, $M \mapsto |\det\, M|^{p/d}$ has at most $p$-growth, and therefore, thanks to the assumption on the support of $\nu$ together with Young measure representation applied to the test function
\begin{equation}\label{eq:dettest}
f(A)=\begin{cases}|\det\, A|^{p/d}, & \text{if $\det\, A<0$}\\
0, & \text{otherwise,}
\end{cases}
\end{equation}
we may assume (after passing to a subsequence if necessary) that
\begin{equation}\label{eq:closetoK}
\int_{\{\det\, \nabla u_j<0\}}|\det\, \nabla u_j(x)|^{p/d}\dd x<\frac{1}{j^p}.
\end{equation}
Now apply Proposition~\ref{prop:nonnegapprox} to each $u_j$ (with arbitrary $\epsilon$) to obtain a new sequence $\{v_j\}$, such that the $v_j$ have nonnegative Jacobians a.e., $v_j-u\in\Wrm^{1,p}_0(\Omega,\R^d)$, and, by virtue of~\eqref{eq:closetoK} and part~(iii) in the proposition,
\[
\norm{\nabla u_j-\nabla v_j}_p<\frac{C_p^{1/p}}{j}.
\]  
It follows that $(v_j)$ is $p$-equiintegrable and generates $\nu$ by Lemma~\ref{lem:pclosegeneration}.
\end{proof}

\begin{proof}[Proof of Proposition~\ref{prop:nonnegapprox}]
In the course of this proof we construct a sequence of gradients $\{\nabla v^l\}_{l\in\N}$, bounded in $\Lrm^p(\Omega;\R^{d \times d})$, such that
\begin{equation}\label{eq:smalldet}
\int_{\{\det\, \nabla v^l<0\}}|\det\, \nabla v^l|^{p/d}\dd x\leq 2^{-lp}\int_{\{\det\, \nabla u<0\}}|\det\, \nabla u|^{p/d}\dd x
\end{equation}
and satisfying further properties mentioned in the following. In particular, the sequence is constructed such that all $v^l$ satisfy $v^l-u\in\Wrm^{1,p}_0(\Omega,\R^d)$.

To begin with, we set $\nabla v^0=\nabla u$. If $\nabla v^l\in \Lrm^p(\Omega;\R^{d \times d})$ has already been constructed, we find $\nabla v^{l+1}$ in the following way: by Proposition~\ref{prop:geometry}, for a.e.\ $x \in \Omega$ for which $\det\, \nabla v^l(x)<0$, we can find a $p$-laminate $\nu_x^l$ with support in the set of matrices with $\det\, M=0$ and such that
\[
\int|\frarg|^p\dd\nu_x^l\leq C_p|\nabla v^l(x)|^p
\]  
and $[\nu_x^l]=\nabla v^l(x)$. For $x \in \Omega$ with $\det\, \nabla v^l(x)\geq0$ we simply set $\nu_x^l=\delta_{\nabla v^l(x)}$. Thus we obtain a Young measure $\nu^l$ with $\int_{\Omega}\int|\frarg|^p\dd\nu_x^l\dd x<\infty$ and $[\nu^l]=\nabla v^l$ and the property that $\nu_x^l$ is a $p$-laminate for almost every $x \in \Omega$; in particular, $\nu$ is a gradient $p$-Young measure. Lemma~\ref{lem:KinPedrefined} then gives us a $p$-equiintegrable sequence of gradients $(\nabla v^{l,m})_{m\in\N}$ generating $\nu^l$ such that $v^{l,m}-u\in\Wrm^{1,p}_0(\Omega,\R^d)$. By Young measure representation, again using the test function $f$ from~\eqref{eq:dettest}, and the fact that $\nu^l$ is supported on matrices with nonnegative determinant, we may choose $m$ large enough, say $m=M$, and define $\nabla v^{l+1}:=\nabla v^{l,M}$ such that
\[
\int_{\{\det\, \nabla v^{l+1}<0\}}|\det\, \nabla v^{l+1}|^{p/d}\dd x\leq 2^{-(l+1)p}\int_{\{\det\, \nabla u<0\}}|\det\, \nabla u|^{p/d}\dd x.
\]
Moreover, by taking $M$ even larger if necessary, we can ensure that
\begin{equation}\label{eq:Mchoice}
\int_{\Omega}|\nabla v^{l+1}-\nabla v^l|^p\dd x\leq 2^p\int_{\Omega}\int|A-\nabla v^l|^p\dd\nu_x^l(A) \dd x
\end{equation}
(in fact, the left hand side can be made arbitrarily close to the right hand side of this inequality). Indeed, this follows again from Young measure representation, this time with the test function $|A-\nabla v^l|^p$. By virtue of Lemma~\ref{lem:KinPedrefined} we may also assume
\begin{equation}\label{eq:Alsmall}
\absb{\setb{x\in\Omega}{ \text{$\det\, \nabla v^l\geq0$ and $\nabla v^{l+1}(x)\neq\nabla v^l(x)$}}}<2^{-(l+2)}\epsilon.
\end{equation}
Thus we see that~\eqref{eq:smalldet} is satisfied for $l+1$ and $v^{l+1}-u\in\Wrm^{1,p}_0(\Omega,\R^d)$ holds. This completes the definition of our sequence.

Next, for any $l \in \N$, Proposition~\ref{prop:geometry}~(iv) and property~\eqref{eq:smalldet} yield the estimate 
\[
\begin{aligned}
\int_{\Omega}\int|A-\nabla v^l(x)|^p\dd\nu_x^l(A)\dd x&\leq C_p\int_{\{\det\, \nabla v^l<0\}}|\det\, \nabla v^l|^{p/d}\dd x\\
&\leq C_p2^{-lp}\int_{\{\det\, \nabla u<0\}}|\det\, \nabla u|^{p/d}\dd x.
\end{aligned}
\] 
By~\eqref{eq:Mchoice} we thus have
\begin{equation}\label{eq:cauchy}
\int_{\Omega}|\nabla v^{l+1}-\nabla v^l|^p\dd x\leq C_p2^{-(l-1)p}\int_{\{\det\, \nabla u<0\}}|\det\, \nabla u|^{p/d}\dd x,
\end{equation}
so that $(\nabla v^l)_{l\in\N}$ is seen to be a Cauchy sequence in $\Lrm^p(\Omega;\R^{d \times d})$ and therefore has a strong $\Lrm^p$-limit $\nabla v$. In particular, it holds that $v-u\in\Wrm^{1,p}_0(\Omega,\R^d)$ and (ii) follows. Using the triangle inequality and~\eqref{eq:cauchy}, we have
\[
\begin{aligned}
\norm{\nabla v-\nabla u}_p&\leq\sum_{l=0}^{\infty}\norm{\nabla v^{l+1}-\nabla v^l}_p\\
&\leq C^{1/p}_p\left(\int_{\{\det\, \nabla u<0\}}|\det\, \nabla u|^{p/d}\dd x\right)^{1/p}\sum_{l=0}^{\infty}2^{-(l-1)}\\
&\leq 4C^{1/p}_p\left(\int_{\{\det\, \nabla u<0\}}|\det\, \nabla u|^{p/d}\dd x\right)^{1/p},
\end{aligned}
\] 
which proves (iii). Moreover, observe that the sequence $(\nabla v^l)_l$ is $p$-equiintegrable (since it is Cauchy in $\Lrm^p$), and since $|\det\, \nabla v^l|^{p/d}\leq C|v^l|^p$, also $\{|\det\, \nabla v^l|^{p/d}\}_{l\in\N}$ is equiintegrable. By Vitali's Convergence Theorem, therefore, we find that
\[
\int_{\{\det\, \nabla v<0\}}|\det\, \nabla v(x)|^{p/d}\dd x=0,
\] 
which implies $\det\, \nabla v(x)\geq0$ for a.e.\ $x\in\Omega$, i.e.\ (i).

For (iv), define the sets
\[
A_l=\setb{x\in\Omega}{ \det\, \nabla v^l(x)\geq0}
\]
and 
\[
B_l=\setb{x\in\Omega}{\nabla v^{l+1}(x)\neq\nabla v^l(x)},
\]
so that $|A_l\cap B_l|<2^{-(l+2)}\epsilon$ by~\eqref{eq:Alsmall}.

The set in~(iv) is contained in $\bigcup_{l=0}^{\infty}(A_0\cap B_l)$. Since $A_0\subset A_l\cup(A_0\setminus A_l)$ and in view of our bound for $|A_l\cap B_l|$, we can estimate
\[
\begin{aligned}
\left|\bigcup_{l=0}^{\infty}(A_0\cap B_l)\right|&\leq\left|\bigcup_{l=0}^{\infty}(A_l\cap B_l)\right|+\left|\bigcup_{l=0}^{\infty}((A_0\setminus A_l)\cap B_l)\right|\\
&\leq\frac{\epsilon}{2}+\left|\bigcup_{l=0}^{\infty}(A_0\setminus A_l)\right|.
\end{aligned}
\]
For the second term, observe that, for any $L\in\N$,
\[
\left|\bigcup_{l=0}^{L}(A_0\setminus A_l)\right|\leq\left|\bigcup_{l=0}^{L-1}(A_{l}\setminus A_{l+1})\right|,
\]
as can be shown by induction over $L$ using the elementary inclusion $A_0\setminus A_l\subset\bigcup_{i=1}^{l}(A_{i-1}\setminus A_i)$. This implies the same inequality for $L=\infty$. Finally, we note that $A_l\setminus A_{l+1}\subset A_l\cap B_l$ and therefore
\[
\left|\bigcup_{l=0}^{\infty}(A_{l-1}\setminus A_l)\right|<\frac{\epsilon}{2},
\]
which concludes the proof.
\end{proof}

\section{Strictly orientation-preserving generating sequences}

To prove the full claim of Theorem~\ref{thm:main} we need the following auxiliary result:

\begin{proposition} \label{prop:strictlyposdet}
Let $M_0 \in \R^{d \times d}$. Then, for every $\delta > 0$ there exists a finite-order laminate, represented by a homogeneous Young measure $\nu \in \Mbf(\R^{d \times d})$ (with its support $\supp \nu$ a finite set), such that for every $p \in [1,\infty)$ the following assertions hold:
\begin{itemize}
  \item[(i)] $[\nu] = \displaystyle \int \id \dd \nu = M_0$.
  \item[(ii)] $\supp \nu \subset \setb{ M \in \R^{d \times d} }{ \abs{\det\, M} \geq \delta^d }$ and half the matrices in $\supp \nu$ have positive determinant.
  \item[(iii)] $\displaystyle \int \abs{\frarg}^p \dd \nu \leq 2^{p-1}(\abs{M_0}^p + C_p \delta^p)$.
  \item[(iv)] $\displaystyle \int \abs{A-M_0}^p \dd \nu(A) \leq C_p \delta^p$.
  \item[(v)] If $\abs{\det\, M_0} < \delta^d$, then $\supp \nu \subset \setb{ M \in \R^{d \times d} }{ \abs{\det\, M} < 3\delta(\abs{M_0}+2\delta)^{d-1} }$,
\end{itemize}
where $C_p = C(d,p)$.
\end{proposition}

\begin{proof}
As in the proof of Proposition~\ref{prop:geometry}, we can use the singular value decomposition to write
\[
  M_0 = P \begin{pmatrix} \theta_1 & & \\ & \ddots & \\ & & \theta_d \end{pmatrix} Q^T,
  \qquad P,Q \in \mathrm{SO}(d),
\]
and such that $\abs{\theta_1} \geq \abs{\theta_2} \geq \cdots \geq \abs{\theta_d}$. Let $L \in \{0,\ldots,d\}$ be such that $\abs{\theta_k} \geq \delta$ for $k \leq L$ and $\abs{\theta_k} < \delta$ for $k > L$.

It is easy to see that we can decompose such an $M_0$ along $d-L$ rank-one lines as follows:
\begin{align*}
  M_0 &= \frac{1}{2} \sum_{\pm} \begin{pmatrix} \theta_1 & & & & & & \\ & \ddots & & & & & \\ & & \theta_L & & & & \\ & & & \theta_{L+1} \pm 2\delta & & & \\ & & & & \theta_{L+2} & & \\ & & & & & \ddots & \\ & & & & & & \theta_d \end{pmatrix} \\
  &= \cdots \\
  &= \frac{1}{2^{d-L}} \, \sum \underbrace{\setBB{ M_0 + \sum_{k=L+1}^d s_k 2 \delta(\ee_k \otimes \ee_k) }{ s_{L+1} = \pm, \ldots, \, s_d = \pm }}_{\displaystyle =: \Rcal}.
\end{align*}
Define the corresponding laminate
\[
  \nu := \frac{1}{2^{d-L}} \sum_{M \in \Rcal} \delta_M,
\]
which satisfies $[\nu] = M_0$, i.e.~(i). Now, all singular values of any matrix in the set $\Rcal$ have absolute value at least $\delta$, whence~(ii) follows. Recalling~\eqref{eq:Frobenius}, we see that in every splitting step we move at most a distance of $2\delta$, measured in the Frobenius norm, away from our original matrix $M_0$. Hence,~(iv) and then also immediately~(iii) follow with $C_p = (2\sqrt{d})^p$.

For~(v) it suffices to notice that if $\abs{\det\, M_0} < \delta^d$, then at least one $|\theta_k|$ is less than $\delta$, whence every $M \in \Rcal$ has at least one singular value with absolute value less than $3\delta$. Moreover, for every $k$, $\abs{\theta_k} \leq \abs{M_0}$, measured in the Frobenius matrix norm.
\end{proof}

\begin{proof}[Proof of \enquote{(ii) $\Rightarrow$ (i)} in Theorem~\ref{thm:main}] \label{sc:conv_strictly_op}

Using the result from Section~\ref{sc:conv_weakly_op}, we can assume that there exists a generating sequence $(u_j) \subset \Wrm^{1,p}(\Omega;\R^d)$, that is, $\nabla u_j \toY \nu$, with $u_j-u\in\Wrm^{1,p}_0(\Omega,\R^d)$, the family $(\nabla u_j)$ is $p$-equiintegrable and $\det\, \nabla u_j \geq 0$ almost everywhere.

Fix $j \in \N$. Define for $l = 0,1,\ldots$ the function $u^l_j \in \Wrm^{1,p}(\Omega;\R^d)$ as follows: For a.e.\ $x \in \Omega$ let $u^0_j := u_j$. If $u^l_j$ is already defined, let the set $Z^l \subset \Omega$ contain all $x$ such that $\det\, \, \nabla u^l_j(x) = 0$. Then, for $x \in Z^l$, set $\nu^{l+1}_x$ to be the (finite-order) laminate from Proposition~\ref{prop:strictlyposdet} with $M_0 := \nabla u^l_j(x)$ and $\delta := \delta_{j,l}$ to be determined later, whereas for $x \in \Omega \setminus Z^l$ set $\nu^{l+1}_x := \delta_{\nabla u^l_j(x)}$. Hence, for almost every $x \in Z^l$, $\supp \nu^{l+1}_x \subset \setb{ M \in \R^{d \times d} }{ \abs{\det\, M} \geq \delta_{j,l}^d }$.

By the usual Young measure representation results and Lemma~\ref{lem:KinPedrefined}, we can find $w^{l+1}_j \in (\Wrm^{1,p} \cap \Crm^\infty)(\Omega;\R^d)$ such that $w^{l+1}_j-u\in\Wrm^{1,p}_0(\Omega,\R^d)$,
\begin{align}
  &\nabla w^{l+1}_j = \nabla u^l_j  \quad\text{on a subset of $\Omega \setminus Z^l$ with measure at least $\biggl(1 - \frac{1}{2^{l+1}} \biggr)\abs{\Omega \setminus Z^l}$},  \notag\\
  &\int_{\{\det\, \nabla w_j^{l+1}<0\}}|\det\, \nabla w_j^{l+1}|^{p/d}\dd x \leq 2\int_{Z^l}\int_{\{\det\, M<0\}}|\det\, M|^{p/d}\dd\nu_x^{l+1}(M)\dd x, \label{eq:det_nablaw}
\end{align}
and, using property~(iv) from the preceding lemma,
\begin{align}
  \int_{\Omega} \abs{\nabla w^{l+1}_j(x) - \nabla u^l_j(x)}^p \dd x
    &\leq \int_{\Omega} \int \abs{A - \nabla u^l_j(x)}^p \dd \nu^{l+1}_x(A) \dd x + C_p|\Omega| \delta_{j,l}^p  \notag\\
    &\leq 2 C_p |\Omega|\delta_{j,l}^p.  \label{eq:Cauchy_Zlcomp}
\end{align}
Moreover, owing to the fact that half the matrices in $\supp \nu^{l+1}_x$ for a.e.\ $x \in Z^l$ have positive determinant, we can require
\[
  \absb{\setb{ x \in \Omega }{ \det\, \nabla w^{l+1}_j(x) > 0 }} \geq \biggl(1 - \frac{1}{2^{l+2}} \biggr) \biggl( \abs{\Omega \setminus Z^l} + \frac{1}{2} \abs{Z^l} \biggr).
\]
Indeed, this follows from Young measure representation applied with the indicator function of $\set{M\in\R^{d\times d}}{\det\, M>0}$. Note in particular that this set is open, hence its indicator function is lower semicontinuous, and we may therefore use it as a test function (cf.~\cite{Mull99VMMP}, Remark 1 after Corollary 3.3).

Next, we use Proposition~\ref{prop:nonnegapprox} applied to $w_j^{l+1}$ with $\epsilon=\epsilon_l$ sufficiently small to infer that there is yet another function $u^{l+1}_j \in \Wrm^{1,p}(\Omega;\R^d)$ with $u^{l+1}_j-u\in\Wrm^{1,p}_0(\Omega,\R^d)$, $\det\, \nabla u^{l+1}_j \geq 0$ a.e., and such that
\begin{equation} \label{eq:equal_largeset}
  \nabla u^{l+1}_j = \nabla u^l_j  \quad\text{on a subset of $\Omega \setminus Z^l$ with measure at least $\biggl(1 - \frac{1}{2^l} \biggr)\abs{\Omega \setminus Z^l}$}
\end{equation}
and
\begin{equation} \label{eq:det_pos_largeset}
  \absb{\setb{ x \in \Omega }{ \det\, \nabla u^{l+1}_j(x) > 0 }} \geq \biggl(1 - \frac{1}{2^{l+1}} \biggr) \biggl( \abs{\Omega \setminus Z^l} + \frac{1}{2} \abs{Z^l} \biggr).
\end{equation}

Then, for $Z^{l+1}$ we get
\begin{align*}
  \abs{Z^{l+1}} &= \abs{\Omega} - \absb{\setb{ x \in \Omega }{ \det\, \nabla u^{l+1}_j(x) > 0 }} \\
    &\leq \biggl( 1 - 1 + \frac{1}{2^{l+1}} \biggr) \abs{\Omega \setminus Z^l} + \biggl(1 - \frac{1}{2} + \frac{1}{2^{l+2}} \biggr) \abs{Z^l} \\
    &\leq \frac{1}{2^{l+1}}\abs{\Omega \setminus Z^l} + \frac{1}{2^{l+2}}\abs{Z^l} + \frac{1}{2}\abs{Z^l}\\
    &\leq \frac{\abs{\Omega}}{2^{l+1}} + \frac{1}{2}\abs{Z^l}.
\end{align*}
By iterating the above inequality, one obtains that
\begin{equation*}
\abs{Z^l} \leq \frac{l}{2^l} \abs{\Omega} + \frac{\abs{Z^0}}{2^l}
\end{equation*}
and it is easy to check that $\sum^{\infty}_{l=0}\abs{Z^l} < \infty$.

Next use part~(iii) of Proposition~\ref{prop:nonnegapprox},~\eqref{eq:det_nablaw} and part~(v) of Proposition~\ref{prop:strictlyposdet} to estimate
\begin{align*}
\int_{\Omega}|\nabla u_j^{l+1}-\nabla w_j^{l+1}|^p\dd x&\leq C_p\int_{\{\det\, \nabla w_j^{l+1}<0\}}|\det\, \nabla w_j^{l+1}|^{p/d}\dd x\\
&\leq2C_p\int_{Z^l}\int_{\{\det\, M<0\}}|\det\, M|^{p/d}\dd\nu_x^{l+1}(M)\dd x\\
&\leq2C_p(3\delta_{j,l})^{p/d}\int_{Z^l}(|\nabla u_j^l(x)|+2\delta_{j,l})^{p(d-1)/d}\dd x. 
\end{align*}
Therefore, by choosing $\delta_{j,l}$ sufficiently small, we can ensure (bearing in mind~\eqref{eq:Cauchy_Zlcomp})
\[
\norm{\nabla u_j^{l+1}-\nabla u_j^l}_p\leq\frac{1}{2^{l+1}j}.
\]
This means that, for every $j$, $(\nabla u_j^l)_l$ is a Cauchy sequence, whose limit we denote by $\nabla \tilde{u}_j$. In particular, $\tilde{u}_j-u\in\Wrm^{1,p}_0(\Omega,\R^d)$ and, by the triangle inequality in $\Lrm^p$, we obtain
\[
  \norm{\nabla \tilde{u}_j(x) - \nabla u_j(x)}_p
    \leq \sum_{l=0}^\infty \, \norm{\nabla u^{l+1}_j(x) - \nabla u^l_j(x)}_p
    \leq \frac{1}{j} \sum_{l=0}^\infty \frac{1}{2^{l+1}}
     =\frac{1}{j}.
\]

Hence, $(\nabla \tilde{u}_j)$ is $p$-equiintegrable and generates the same Young measure as $(\nabla u_j)$ by Lemma~\ref{lem:pclosegeneration}. It remains to show that $\det\, \nabla\tilde{u}_j>0$ a.e.~in $\Omega$. For this, it suffices to prove that the set
\[
N:=\setb{x\in\Omega}{ \text{$\forall L \in \N \; \exists M \geq L$ such that $\nabla u_j^M(x)\neq\nabla u_j^{M+1}(x)$}}
\]
has zero measure. Indeed, if this is true, there exists a null set $\Gamma$ such that for every $x\in \Omega \setminus \Gamma$, there is $L \in \N$ with $\nabla\tilde{u}_j(x) = \lim_{l\to\infty}\nabla u_j^l(x)=\nabla u_j^{L}(x)$. This follows from the strong convergence of $\nabla u_j^l$ to $\nabla\tilde{u}_j$ in $\Lrm^p$ and the fact that the union of two null sets is null. Thus, 
\begin{equation*}
\setb{x\in\Omega\setminus\Gamma}{\det\, \nabla\tilde{u}_j(x)=0}\subset \bigcup_{L=0}^{\infty}\left(\bigcap_{l=L}^{\infty} Z^l\right) = \liminf_{l\to\infty}\, Z^l \subset \limsup_{l\to\infty}\, Z^l.
\end{equation*}
But $\sum_{l=0}^{\infty}|Z^l|<\infty$ and, by the Borel--Cantelli lemma, $|\limsup_{l}\, Z^l| = 0$. It follows that the set $\setb{x\in\Omega}{\det\, \nabla\tilde{u}_j(x) > 0}$ has full measure.

Finally, to show that $N$ is a null set, observe that by our estimate for $|Z^l|$ and by~\eqref{eq:equal_largeset} we have
\begin{align*}
\absb{\setb{x\in\Omega}{\nabla u_j^M(x)\neq \nabla u_j^{M+1}(x)}}&\leq |\Omega|-\left(1-\frac{1}{2^{M-1}}\right)|\Omega\setminus Z^M|\\
&\leq |\Omega|\frac{M+2}{2^M}.
\end{align*}
Since this is summable in $M$, our claim that $N$ is a null set follows by another application of the Borel--Cantelli lemma. 
\end{proof}

\section{Relaxation of integral functionals}\label{sc:relaxation}

Apart from a characterization of gradient $p$-Young measures, $p<d$, Theorem~\ref{thm:main} can be used to provide a relaxation result of integral functionals in $\Wrm^{1,p}$ under the additional constraint on the admissible deformations that they are orientation-preserving, see Theorem~\ref{thm:relaxation} below. As discussed in the introduction, this is an important requirement in applications. 

Consider the following two functionals for a Carath\'{e}odory function $f \colon \Omega \times \R^{d \times d} \to \R$ and a function $\bar{u}\in\Wrm^{1,p}(\Omega)$:
\begin{itemize}
\item $I(u):= \displaystyle \int_{\Omega}f(x,\nabla u(x)) \dd x$,\qquad defined over the set
\[
\mathcal{A}:=\setb{u\in \Wrm^{1,p}(\Omega,\mathbb{R}^d)}{u\vert_{\partial\Omega}=\bar{u},\,\det\, \nabla u(x)>0\mbox{ a.e.}},
\]
\item $I^{YM}(\nu):= \displaystyle\int_{\Omega} \int f(x,A) \dd \nu_x(A) \dd x$,\qquad defined over the set
\[
\mathcal{A}^{YM}:=\setb{\nu\,\mbox{$p$-GYM}}{ \text{$\supp{\nu_x}\subset\{\det\, M\geq 0\}$ a.e.\ and $[\nu]=\nabla u$, $u\vert_{\partial\Omega}=\bar{u}$}},
\]
\end{itemize}
where we used \enquote{$p$-GYM} as an abbreviation for \enquote{gradient $p$-Young measure}. We restate Theorem~\ref{thm:relaxationintro} for the reader's convenience:

\begin{theorem}\label{thm:relaxation}
Suppose that $\Omega\subset\mathbb{R}^d$ is a bounded Lipschitz domain, $\bar{u}\in\Wrm^{1,p}(\Omega)$, and $f: \Omega \times \mathbb{R}^{d\times d}\rightarrow\mathbb{R}$ is a Carath\'{e}odory function satisfying
\[
c(\abs{A}^p-1)\leq f(x,A)\leq C(1+\abs{A}^p)
\]
for all $(x,A)\in\Omega\times\mathbb{R}^{d\times d}$, some $p\in (1,d)$, and constants $0<c\leq C$. Then,
\[
\inf_{\mathcal{A}}\, I=\min_{\mathcal{A}^{YM}}\, I^{YM}.
\]
In particular, whenever $(u_j)$ is an infimizing sequence of $I$ in $\mathcal{A}$, a subsequence of $(\nabla u_j)$ generates a Young measure $\nu\in\mathcal{A}^{YM}$ minimizing $I^{YM}$ in $\mathcal{A}^{YM}$. Conversely, whenever $\nu$ minimizes $I^{YM}$ in $\mathcal{A}^{YM}$, there exists an infimizing sequence $(u_j)$ of $I$ in $\mathcal{A}$ such that $(\nabla u_j)$ generates $\nu$.
\end{theorem}

\begin{proof}
Note that by standard arguments $\min_{\mathcal{A}^{YM}}\, I^{YM}$ exists. Also, for each $u\in\mathcal{A}$,
\[
I(u)=I^{YM}(\delta_{\nabla u})\geq\min_{\mathcal{A}^{YM}}\, I^{YM}
\]
and hence,
\begin{equation}
\label{eq:relaxation1}
m:=\inf_{\mathcal{A}}\, I\geq\min_{\mathcal{A}^{YM}}\, I^{YM}=:m^{YM}.
\end{equation}
Now let $\nu\in\mathcal{A}^{YM}$ such that $I^{YM}(\nu)=m^{YM}$. By Theorem~\ref{thm:main} there exists a sequence $(u_j)\subset\mathcal{A}$ such that $(\nabla u_j)$ generates $\nu$ and $(\nabla u_j)$ is $p$-equiintegrable. Then
\begin{align*}
m=\inf_{\mathcal{A}}\, I &\leq \lim_{j}\, I(u_j)
 = \lim_j \int_{\Omega}f(x,\nabla u_j(x))\dd x \\
 &= \int_{\Omega} \int f(x,A) \dd \nu_x(A) \dd x = \min_{\mathcal{A}^{YM}}\, \, I^{YM}=m^{YM}. 
\end{align*}
In particular, by~\eqref{eq:relaxation1}, $I(u_j)\rightarrow m$, as $j\rightarrow\infty$, i.e.~$u_j$ is infimizing for $I$ in $\mathcal{A}$ and
\begin{equation}
\label{eq:relaxation3}
m=m^{YM}.
\end{equation}
Conversely, let $(u_j)\subset\mathcal{A}$ such that $I(u_j)\to m$, as $j\to\infty$. Then, by Theorem~\ref{thm:main}, a subsequence of $(\nabla u_j)$ generates a Young measure $\nu\in\mathcal{A}^{YM}$ and it suffices to show that $I^{YM}(\nu)=m^{YM}$. But, since $f$ is continuous and bounded below, by a standard result (see e.g.~Theorem 6.11, \cite{Pedr97PMVP})
\[
I^{YM}(\nu) =\int_{\Omega} \int f(x,A) \dd \nu_x(A) \dd x
 \leq \liminf_{j}\int_{\Omega}f(x,\nabla u_j(x))\dd x
 = m = m^{YM}
\]
by~\eqref{eq:relaxation3} and the proof is complete.
\end{proof}

\section{Remarks on the integrability constraint and convex hulls} \label{sc:remarks}

\subsection{Rigidity versus softness}

Assume that $p \geq d$ where $d$ denotes the dimension. Then, there cannot exist a sequence of gradients $(\nabla u_j) \subset \Lrm^p(\Omega;\R^d)$ generating a given Young measure $\nu$ satisfying the properties~(I)--(IV) and such that every $u_j$ exhibits the same boundary values as its $\Wrm^{1,p}$-weak limit, $\det\, \nabla u_j > 0$ a.e., and $(\nabla u_j)$ is uniformly bounded in $\Lrm^p$. This can be seen easily, for instance by taking $\nu_x := \delta_0$ a.e.: If a sequence $(\nabla u_j)$ with the above properties existed, then
\[
  \int_\Omega \det\, \nabla u_j \dd x = 0
\]
because the determinant function is quasi-affine, $|\det \,A|\leq C|A|^d$ and the boundary condition $\nabla u_j|_{\partial \Omega} = 0$ holds. On the other hand, since $\det\, \nabla u_j > 0$ a.e.,
\[
  \int_\Omega \det\, \nabla u_j \dd x > 0,
\]
a contradiction. Of course, this argument even applies to single functions, not necessarily to sequences.

In the language of convex integration, for $p \geq d$ the property of having positive Jacobian is \enquote{rigid} for gradients $\nabla u \in \Lrm^p(\Omega;\R^{d \times d}$). In particular, a function satisfying this constraint approximately cannot be improved to satisfy it strictly by changing the function only \enquote{slightly} (to the order of how well the constraint is already satisfied).

Our Theorem~\ref{thm:main} contrasts this rigidity statement with the assertion that for $p < d$ the situation is indeed \enquote{flexible}, i.e.\ the improvement to strictly satisfying the positive Jacobian constraint is possible.

This phenomenon is in fact already present for Proposition~\ref{prop:geometry}: There, we construct a sequence of finite-order laminates $\nu_j$ such that
\[\nu_j\toweakstar\nu.\]
Each $\nu_j$ is a gradient $\infty$-Young measure but the supports are not uniform and we cannot conclude that $\nu$ is a gradient $\infty$-Young measure. However, property (iii) states that
\[\int\abs{\frarg}^p\dd\nu_j\leq C\]
for a universal constant $C$ and all $j$, hence $\nu$ is a gradient $p$-Young measure. By the Kinderlehrer--Pedregal characterization of gradient $p$-Young measures, the fact that the determinant is polyconvex, and $|\det\, A|\leq C\abs{A}^d$, where $p\geq d$, one would obtain that
\[\int \det \dd \nu_x=\det\, [\nu]=\det\, M_0<0.\]
But this contradicts (ii), that is $\supp \nu \subset \setb{ M \in \R^{d \times d} }{ \det\, M \geq 0 }$.

More generally, one cannot prove a statement like Proposition~\ref{prop:geometry} for $p\geq d$; due to the above reasoning, any gradient $p$-Young measure $\nu$ supported entirely on matrices with non-negative determinant cannot satisfy $[\nu]=M_0$ where $\det\, M_0<0$.

Nevertheless, in our result the restriction $p<d$ only appears as a restriction on the orientation-preserving sequence generating a given measure $\nu$ and not on $\nu$ itself, i.e.~$\nu$ may be a gradient $q$-Young measure with $q\geq d$ but the orientation preserving maps generated are only uniformly bounded in $\Wrm^{1,p}(\Omega;\R^d)$ for $p < d$. We note that a similar situation occurs in the characterization of gradient Young measures generated by gradients of $K$-quasiregular mappings in $d=2$, see~\cite{AstFar02QR}. In particular, for any gradient $q$-Young measure, $q>2K/(K+1)$, the generating sequence lies in general only in $\Wrm^{1,p}(\Omega;\R^2)$ ($\Omega\subset\R^2$) for $p < 2K/(K-1)$. The case of orientation-preserving maps corresponds to the limit $K\to\infty$, whence $p<2$.

It is also worth noting that our proofs provide a very general, yet abstract, counterexample on the weak continuity of the determinant, see e.g.~\cite{BalMur84WQVP} for such examples, \cite{HenMor10IWCD} for examples in the context of cavitation and the work in~\cite{KaKrKr12SWCN} on the weak continuity of null Lagrangians at the boundary. In particular, let $p<d$, $q\geq p$ and $u\in\Wrm^{1,q}(\Omega;\R^d)$ such that $\det \, \nabla u(x) < 0$ a.e.~in $\Omega$. By Proposition~\ref{prop:geometry}, for a.e.~$x\in\Omega$, there exists a homogeneous gradient $p$-Young measure $\nu_x$ supported in the set $\setb{M\in\R^{d\times d}}{\det\,M\geq0}$ with $[\nu_x]=\nabla u(x)$ and
\[
\int |\frarg|^p\dd\nu_ x\leq C_p|\nabla u(x)|^p.
\]
Then the family of measures $\nu=(\nu_x)_{x\in\Omega}$ satisfies properties (I)-(IV) and, by our methods, we can extract a sequence $(u_j)\subset\Wrm^{1,p}(\Omega,\R^d)$ such that $u_j\rightharpoonup u$ in $\Wrm^{1,p}(\Omega,\R^d)$ and $\det\,\nabla u_j(x) \geq 0$ (even strict inequality) for all $j\in\N$ and a.e.~$x\in\Omega$.

\subsection{Convex hulls}

Finally, we make a few remarks about different convex hulls, cf.~\cite{Daco08DMCV,Mull99VMMP,Zhang00rank,Yan07pQChulls} (we follow the notation of the latter reference). This will also clarify the relationship between our construction and \enquote{classical} convex integration as in~\cite{MulSve03CILM}.

Let $D\subset\R^{d\times d}$ be closed and define the set 
\[
\Qcal:=\setb{ f:\R^{d\times d}\rightarrow\R }{ \text{$f$ quasiconvex} } 
\]
and, respectively, the set
\[
\Qcal_p:=\setb{ f\in \Qcal }{ \text{ $| f(A)| \leq c(1+|A|^p)$ for some constant $c>0$} }.
\]
The \term{quasiconvex hull} of $D$ is then defined as
\[
D^{qc}:=\setb{ M\in\R^{d\times d}}{ \text{$f(M)\leq\sup_{D} f$ for all $f\in \Qcal$} }.
\]
Similarly, one may define the \term{$p$-quasiconvex hull} of $D$, often referred to as the \enquote{strong} $p$-quasiconvex hull, as
\[
D^{qc}_p:=\setb{ M\in\R^{d\times d}}{ \text{$f(M)\leq\sup_{D} f$ for all $f\in \Qcal_p$} }.
\]
Trivially, $D^{qc}\subset D^{qc}_p$ and when $D$ is compact the reverse inclusion also holds, so that $D^{qc}=D^{qc}_p$.

In terms of Young measures, let us define
\[
D^{YM}:=\setb{ M\in\R^{d\times d} }{ \text{$M=[\nu]$ for some $\infty$-HGYM $\nu$ with $\supp\nu\subset D$} }
\]
and
\[
D^{YM}_p:=\setb{ M\in\R^{d\times d} }{ \text{$M=[\nu]$ for some $p$-HGYM $\nu$ with $\supp\nu\subset D$} },
\]
where \enquote{$p$-HGYM} stands for \enquote{homogeneous gradient $p$-Young measure}. Again, one has $D^{YM}\subset D^{YM}_p$ and trivially $D^{YM}\subset D^{qc}$ and $D^{YM}_p\subset D^{qc}_p$ (in fact $D^{YM}_p = D^{qc}_p$ \cite{Yan07pQChulls}). Moreover, for $D$ compact, $D^{qc}=D^{YM}=D^{YM}_p=D^{qc}_p$.

In our context, let
\[
  D := \setb{ M \in \R^{d \times d} }{ \det\, M \geq 0 }.
\]
We see that $D$ is a sublevel set of the polyconvex (hence quasiconvex and rank-one convex) function $-\det$. In particular, this implies that $D$ is polyconvex, quasiconvex and rank-one convex and
\[
  D=D^{qc}.
\]

On the other hand, our geometric Proposition~\ref{prop:geometry} implies that for any matrix $M$ with $\det\, M < 0$, there exists a homogeneous gradient $p$-Young measure ($p<d$) supported in $D$ with barycenter $M$. Trivially, for any $M$ with $\det\, M\geq 0$, $\delta_M$ is the corresponding homogeneous gradient $p$-Young measure, i.e.\
\[
  D^{YM}_p=\R^{d\times d}.
\]
Then one obtains that for all $p<d$,
\[D=D^{qc}\subset D^{qc}_p=\R^{d\times d},\]
providing an example of a non-compact set for which $D^{qc}\neq D^{qc}_p$ for all $p<d$. A similar phenomenon also occurs for the quasiconformal set where the critical exponent is $p=dK/(K+1)$, see~\cite{Yan01SChullsQCsets}.

As a further illustration, consider the application of the geometric proposition to a matrix $M \in \R^{d \times d}$ with $\det\, M < 0$. For every finite-order laminate $\nu_j$ ($j \in \N$) in the iterative construction we have
\[
  \int -\det\, A \dd \nu_j(A) = -\det\, M > 0
\]
because the determinant function is linear along rank-one lines (along which we split). However, because of the $p$-growth, the preceding assertion is \emph{lost} in the limit (since the support of the $p$-laminate is in $D$):
\[
  \int -\det\, A \dd \nu(A) = 0.
\]
Hence, the construction in Proposition~\ref{prop:geometry} leads out of the classical lamination convex hull.

We remark that in \enquote{classical} convex integration---strictly interpreted---one writes a matrix in the rank-one convex hull $D^{rc}$ of a set $D$ as a laminate supported on $D$ itself, but as explained above, in our situation this is of no use. We end by remarking that the general convergence principle in Sections~\ref{sc:conv_weakly_op},~\ref{sc:conv_strictly_op} might also be transferable to other constraints $\nabla u \in D$ if $D^{qc}_p = \R^{d \times d}$ and if similar good estimates to the ones in Proposition~\ref{prop:geometry} hold.

\bibliography{../Bib/OrientPresYM.bib}
\bibliographystyle{amsalpha}


\end{document}